\documentclass[submission,,copyright,creativecommons]{eptcs}
\providecommand{\event}{AiML 2026} 

\usepackage{aiml26}
\usepackage{iftex}

\ifpdf
  \usepackage{underscore}         
  \usepackage[T1]{fontenc}        
\else
  \usepackage{breakurl}           
\fi

\usepackage{graphicx} 
\usepackage{amsmath}
\usepackage{amssymb,amsthm}
\usepackage{xcolor}
\usepackage{linguex}
\usepackage{multicol}
\usepackage{amsthm}

\newcommand{\threeSAT}{3\mathrm{SAT}}
\newcommand{\threeUNSAT}{3\mathrm{UNSAT}}
\newcommand{\M}{\mathcal{M}}

\newcommand{\ESO}{\mathrm{ESO}}

\newcommand{\inqBT}{\ensuremath{\mathsf{InqBT}}}
\newcommand{\inqBTe}{\ensuremath{\mathsf{InqBT}+[x]}}
\newcommand{\inqBQ}{\ensuremath{\mathsf{InqBQ}}}
\newcommand{\dep}[2]{=\hspace{-0.1cm}(#1,#2)}

\newcommand{\existsi}{\mathord{\exists\hspace{-.4em}\exists}} 

\newcommand{\lori}{\,\disji\,}
\def\disji{\rotatebox[origin=c]{-90}{$\!{\geqslant}$}}

\newcommand{\N}{\mathbb{N}}

\title{On the Expressive Power of Inquisitive Team Logic and Inquisitive First-Order Logic}
\author{Juha Kontinen
\institute{University of Helsinki\\ Helsinki, Finland}
\email{juha.kontinen@helsinki.fi}
\and
Ivano Ciardelli 
\institute{University of Padua\\
Padua, Italy}
\email{ivano.ciardelli@unipd.it}
}

\newcommand{\titlerunning}{On the Expressive Power of Inquisitive Team Logic and Inquisitive First-Order Logic}
\newcommand{\authorrunning}{J. Kontinen \& I. Ciardelli}

\hypersetup{
  bookmarksnumbered,
  pdftitle    = {\titlerunning},
  pdfauthor   = {\authorrunning},
  pdfsubject  = {EPTCS},               
  pdfkeywords = {keyword1, keyword2} 
}

\begin{document}
\maketitle

\begin{abstract}
Inquisitive team logic is a variant of inquisitive logic interpreted in team semantics, which has been argued to provide a natural setting for the regimentation of dependence claims. With respect to sentences, this logic is known to be expressively equivalent with first-order logic.
In this article we show that, on the contrary, the expressive power of open formulas in this logic properly exceeds that of  first-order logic. On the way to this result, we show that if inquisitive team logic is extended with the range-generating universal quantifier adopted in dependence logic, the resulting logic can express finiteness; as a consequence, this logic is not compact and has non-arithmetic complexity. We further extend our results to standard inquisitive first-order logic, showing that some sentences of this logic express non first-order properties of models, thus settling an open problem posed in \cite{Ciardelli:23book}.
\end{abstract}

\section{Introduction}

The last two decades have seen the rise of many logics based on team semantics \cite{Hodges:97}, a mathematical framework for studying concepts and phenomena that arise in contexts involving a plurality of data such as databases or probability distributions. Logics based on  team semantics have found applications in various  fields, including   database theory \cite{Kontinen:2013:independence,Hannula:2020:polyteam}, Bayesian networks and probabilistic dependencies \cite{Corander:2019,Hirvonen24}, quantum foundations \cite{Durand:2018:probabilistic,Abramsky:2021:team,AlbertG22}, as well as  formal semantics of natural language  \cite{Ciardelli:18book,hawke,aloni2022}.

In  team semantics,  formulas are interpreted over a single first-order structure together with \emph{a set of assignments} (aka a \emph{team}) rather than single assignment as in Tarskian semantics. A prominent example of a team-based logic is dependence logic, introduced in  \cite{Vaananen:2007:dependence}, which may be seen as an extension of standard first-order logic with dependence atoms $\dep{\vec{x}}{y}$, expressing the fact that the values of the variables $\vec{x}$ functionally determine the value of $y$. Inclusion logic \cite{Galliani:2012} and independence logic \cite{Gradel:2013} are two other extensively studied team-based logics that extend first-order logic with inclusion and independence, corresponding to  inclusion and embedded multivalued dependencies in  database theory  \cite{Engstrom12}.

Since the introduction of dependence logic, the expressivity and  complexity properties of logics in team semantics have been extensively studied (see, e.g., \cite{hannula2018,Durand:2022:tractability,DurandKV24}). In particular, the interesting discrepancy between sentences and open formulas of dependence logic  is now well understood \cite{KontinenV09}. With respect to sentences,  dependence logic is equi-expressive with existential second-order logic (ESO); however, for open formulas the correspondence is not as simple. In particular, formulas $\phi(x_1,\ldots,x_n)$ of dependence logic can be translated to  $\ESO$-sentences  $\psi(R)$ which can refer to a  team over $\{x_1,\ldots,x_n\}$ via an extra $n$-ary relation symbol $R$ occurring  \emph{only negatively} in   $\psi(R)$. The converse also holds, i.e., for any such $\ESO$-sentence  $\psi(R)$ there exists a  formula $\phi(x_1,\ldots,x_n)$ of dependence logic such that  for all first-order models $\M$ and non-empty teams (i.e., sets of assignments) $X$ over the variables  $\{x_1,\ldots,x_n\}$:
\begin{equation}\label{eq1}
 \M \models _X \phi(x_1,\ldots,x_n) \iff (\M,X[\vec{x}])\models \psi (R),  
 \end{equation}
where $X[\vec{x}]=\{ (s(x_1),\ldots, s(x_n)) \, |\, s\in X  \}$ is the relation encoding the team  $X$.

Our primary focus in this paper is an interesting team-based logic, closely related to inquisitive logic \cite{Ciardelli:18book,Ciardelli:23book}, a research program that aims to extend the scope of logic to include not only statements, but also questions. The standard system of inquisitive first-order logic, called $\inqBQ$ \cite{Ciardelli:23book},\footnote{In the inquisitive semantics literature, the so-called ``basic'' system of inquisitive propositional logic is denoted \textsf{InqB}; the name $\inqBQ$ reflects the fact that this system is obtained by enriching \textsf{InqB} with Quantifiers.} can be seen as an extension of standard first-order logic with a question-forming disjunction, $\lori$, and a question-forming existential, $\existsi$. Using these operators, one can express, in addition to statements like ``all objects are $P$'' (expressed as usual by $\forall xPx$) also questions such as ``whether or not all objects are $P$'' ($\forall xPx\lori\neg\forall xPx$), ``which objects are $P$'' ($\forall x(Px\lori\neg Px)$), and ``what is one example of a $P$'' ($\existsi xPx$). 

The logic $\inqBQ$ is interpreted in terms of an intensional semantics involving multiple possible worlds, each capturing a state of affairs and corresponding formally to a standard relational structure for predicate logic. However, the very same language can also be interpreted relative to teams, leading to a system known as \emph{inquisitive team logic}, $\inqBT$ \cite{Ciardelli:23book,CiardelliConti:26}.\footnote{This system was first considered by Yang  \cite{yang2014} under the name \textsf{WID} (for weak intuitionistic dependence logic).} In this system, one may express questions about the values of variables, including in particular the question ``what is the value of $x$?'', which may be expressed by the formula $\mu x:=\forall y({y=x}\,\lori\, {y\neq x})$ (or, equivalently, by the formula $\lambda x:=\existsi y(y=x)$). Dependencies can be expressed in this system as implications among questions: in particular, the dependence atom $\dep{\vec{x}}{y}$ can be defined in $\inqBT$ by the formula $\mu x_1\land\dots\land\mu x_n\to \mu y$.
While the logics $\inqBQ$ and $\inqBT$ are different, they are related by entailment-preserving translations that often allow a transfer of results between them \cite{CiardelliConti:26}.

In spite of many investigations, which have led to important results (\cite{Grilletti:19,Grilletti:20,Grilletti:21,CiardelliGrilletti:22,GrillettiCiardelli:23}), some key questions about the meta-theoretic properties of the (semantically defined) inquisitive logics $\inqBT$ and $\inqBQ$ remain open. It is not known, e.g., whether the sets of validities of these logics are recursively enumerable (and, thus, if a complete axiomatization is possible), nor whether these logics are entailment-compact, in the sense that a conclusion follows from a set of premises only if it follows from a finite subset of these premises. 

In this paper, we investigate the expressive power of the logics \inqBT\ and \inqBQ. In the case of \inqBT, a simple argument going back to \cite{yang2014} shows that with respect to sentences, the expressive power of this system coincides with that of first-order logic. The question about the expressive power of open formulas, however, is currently open. In particular,
it is not known whether open formulas of $\inqBT$ can be systematically translated to sentences of first-order logic in a language extended with a relation symbol referring to the team.

\begin{itemize}
\item \textbf{Open Question 1.} Given a formula $\phi(x_1,\dots,x_n)$ of $\inqBT$, is there always a first-order sentence $\psi(R)$, in a signature extended with an $n$-ary relation symbol $R$, which is equivalent to $\phi(x_1,\dots,x_n)$ in the sense of equation (\ref{eq1}) above?
\end{itemize}

\noindent
A wide range of particular examples admitting a first-order translation may suggest a positive answer to this question. We show that, on the contrary, the answer is negative: there are open formulas of $\inqBT$ that do not have a first-order counterpart---expressing genuinely second-order properties. 

On the way to this result, we consider a richer logic, denoted $\inqBT+[x]$. In $\inqBT$, the two quantifiers $\forall x$ and $\existsi x$ have a simple semantics: they explore different ways of assigning to $x$ a constant value throughout the team (in the team semantics literature, these quantifiers are standardly denoted by means of the notation $\forall^1 x$ and $\exists^1 x$, see for instance \cite{KontinenV09}). In the logic $\inqBT+[x]$, we add to these quantifiers a third quantifier, $[x]$, corresponding to the universal quantifier adopted in dependence logic (and denoted in that context simply as $\forall x$): semantically, this quantifier expands the team with all possible values for~$x$.

The reasons to be interested in the logic $\inqBT+[x]$ are not merely technical. In recent work \cite{Ciardelli:26}, it has been argued that the quantifiers $\forall x$ and $[x]$ reflect two different ways to make general claims in natural language: an ``extensional'' way, corresponding to universal claims like (a) below, and an ``intensional'' way, corresponding to generic claims like (b).

\ex.[]
\a.[(a)] Every triangle has angles that sum to $360^\circ$
\b.[(b)] A triangle has angles that sum to $360^\circ$

The basic idea is that, while (a) is evaluated by considering many particular triangles, (b) is evaluated by considering a single but completely generic triangle. 
In \cite{Ciardelli:26}, the relationship between these two ways of expressing generality has been studied systematically in the context of a logic, $\textsf{InqWT}+[x]$, which is nothing but the $\existsi$-free fragment of $\textsf{InqBT}+[x]$. A central result of \cite{Ciardelli:26} is that every sentence of this logic can be translated (albeit in a rather complex way) into an equivalent sentence of standard predicate logic. It is natural to wonder whether such a result extends to the richer system $\textsf{InqBT}+[x]$, including the inquisitive existential quantifier $\existsi$.

\begin{itemize}
\item \textbf{Open Question 2.}\\ 
Is every sentence of $\inqBT+[x]$ equivalent to some sentence of standard first-order logic?
\end{itemize}
Again, in this paper we answer the question in the negative. We show that in $\inqBT+[x]$ we can write a sentence expressing the fact that the model is finite (or infinite)---a property which is famously not expressible in first-order logic. Using this result, we show that the logic $\inqBT+[x]$ is not compact (even in the weaker sense of satisfiability), and that it is not recursively axiomatizable (indeed, we show that the set of its validities is not even arithmetical). 

The sentence of $\inqBT+[x]$ which is shown to express finiteness has the form $[x][y]\phi(x,y)$, where $\phi(x,y)$ is an open formula of $\inqBT$. It is this formula $\phi(x,y)$ which we show not to be expressible in first-order logic augmented with a binary predicate $R$ for the team. 

Finally, we adapt our results to settle a related open question about the expressive power of the standard system of inquisitive first-order logic, $\inqBQ$. A model $M=(W,D,I)$ for \inqBQ\ consists of a domain $W$ of possible worlds, a domain $D$ of individuals, and a world-relative evaluation function $I$, which assigns to each world $w$ a standard relational structure $M_w$ over $D$. Such a model can be encoded as a two-sorted structure $M^*=(W,D,I^*)$ for a signature in which each function or relation symbols is expanded with an extra argument for the world of evaluation. A natural question is whether, modulo this encoding, there exists a translation from \inqBQ\ into two-sorted first-order logic, in the following sense.

\begin{itemize}
\item \textbf{Open Question 3.}  
Given a sentence $\phi$ of $\inqBQ$, is there always a corresponding sentence $\phi^*$ of two-sorted first-order logic (over a suitably translated signature incorporating world-dependency into the predicate and function symbols) such that for every model $M$, $M$ satisfies $\phi$ in \inqBQ\ if and only if $M^*$ satisfies $\phi^*$ in predicate logic?
\end{itemize}

\noindent
This question is explicitly posed as an open problem in \cite{Ciardelli:23book}. A positive answer is known for two important fragments of \inqBQ: the \emph{classical antecedent (clant) fragment} \cite{Grilletti:21,Ciardelli:23book}, where antecedents are restricted to formulas of standard first-order logic, and the \emph{restricted existential (rex) fragment} \cite{CiardelliGrilletti:22}, where inquisitive existentials can occur only within conditional antecedents. Using an adaptation of the proof we develop for \inqBT, we show that in the general case, the answer to the question is negative: there are formulas of \inqBQ\ which have no first-order counterpart, expressing genuinely second-order properties.

\medskip
\noindent
{\bf Organization.}
The paper is structured as follows. In \S\ref{Prels} we cover the necessary background on $\inqBT$ and $\inqBT+[x]$. In \S\ref{results} we prove the novel results about these logics, settling Open Questions 1 and 2. \S\ref{sec:inqbq} extends the scope to \inqBQ, settling Open Question 3. \S\ref{conc} concludes with a summary and open problems.

\section{Preliminaries}\label{Prels}

In this section, we introduce the inquisitive team logic \inqBT\ and its extension \inqBTe\ which are the subjects of this paper. We also mention some of their basic properties. We omit the proofs of the results in this section, since they are straightforward and standard in the literature (see, e.g., \cite{Ciardelli:23book}).

The syntax of \inqBT\ and \inqBTe, and the classical fragment of these logics, are given by the following BNF definitions, where $p$ denotes an atom of predicate logic which, as usual, can have the form $(t_1=t_2)$ or the form $P(t_1,\dots,t_n)$, where $P$ is an $n$-ary predicate and $t_1,t_2,\dots,t_n$ are terms.
\begin{align*}
&\text{Syntax of \inqBT} && \phi\;::=\; p\mid \bot\mid (\phi\land\phi)\mid (\phi\lori\phi)\mid(\phi\to\phi)\mid \forall x\phi\mid\existsi x\phi\\
&\text{Syntax of \inqBTe} && \phi\;::=\; p\mid \bot\mid (\phi\land\phi)\mid (\phi\lori\phi)\mid(\phi\to\phi)\mid \forall x\phi\mid\existsi x\phi\mid[x]\phi\\
&\text{Classical formulas} &&\phi\;::=\; p\mid \bot\mid (\phi\land\phi)\mid(\phi\to\phi)\mid \forall x\phi
\end{align*}

The set of classical formulas can be identified with the standard language of first-order predicate logic, with a particular choice of primitives. The remaining operators can be defined in a standard way by letting $\neg\phi:=(\phi\to\bot)$; $\phi\lor\psi:=\neg(\neg\phi\land\neg\psi)$; $\phi\leftrightarrow\psi:=(\phi\to\psi)\land(\psi\to\phi)$; $\exists x\phi:=\neg\forall x\neg\phi$.
In addition, it is standard in inquisitive logic to define a question-mark operator as $?\phi:=(\phi\lori\neg\phi)$.

The semantics of these systems is given in terms of a relation of \emph{support}, which is defined relative to a model and a team. A model is defined, as usual in predicate logic, as a pair $\M=(D,I)$ consisting of a non-empty domain $D$ and an interpretation function $I$ assigning to each predicate or function symbol in the vocabulary a relation or function over $D$ with the corresponding arity. A team is defined as a set of assignments defined over a common domain of variables; we use the meta-variables $X,Y,\dots$ for teams, and indicate by $\text{dom}(X), \text{dom}(Y),\dots$ the corresponding domains, i.e., the sets of variables on which the assignments in the teams are defined. If $\{x_1,\dots,x_n\}\subseteq\text{dom}(X)$, we denote by $X[x_1,\dots,x_n]$ the $n$-ary relation on $D$ defined by
$$X[x_1,\dots,x_n]=\{(g(x_1),\dots,g(x_n))\mid g\in X\}.$$
Note that, if $x_1,\dots,x_n$ is a finite sequence of variables, there is a one-to-one correspondence between teams defined on the set $\{x_1,\dots,x_n\}$ and $n$-ary relations on $D$, given by the map $X\mapsto X[x_1,\dots,x_n]$.

To state the semantics, we will make use of two operations on teams. First, if $g$ is an assignment, $x$ is a variable, and $d$ an object from the domain $D$ of the model, we denote by $g[x\mapsto d]$ the assignment over the variables $\text{dom}(g)\cup\{x\}$ which maps $x$ to $d$ and agrees with $g$ on the remaining variables in $\text{dom}(g)$. Now if $X$ is a team, $x$ a variable, and $d\in D$ an object, we denote by $X[x\mapsto d]$ the team over the domain $\text{dom}(X)\cup\{x\}$ obtained by assigning to $x$ the constant value $d$ through the team:
$$X[x\mapsto d]=\{g[x\mapsto d]\mid g\in X\}.$$
Similarly, if $X$ is a team, $x$ a variable, and $A\subseteq D$ a set of objects, we denote by $X[x\mapsto A]$ the team over $\text{dom}(X)\cup\{x\}$ obtained by assigning to $x$ each possible value from $A$, in combination with all values for the remaining variables represented in $X$:
$$X[x\mapsto A]=\{g[x\mapsto d]\mid g\in X,d\in A\}=\bigcup_{d\in A}X[x\mapsto d].$$
If $t$ is a term, its denotation $[t]_\M^g$ relative to a model $\M$ and an assignment $g$ defined on the variables occurring in $t$ is defined in the usual way. We can now state the semantics of our logics as follows.

\begin{definition} The relation of \emph{support} $\M\models_X\phi$ between a model $\M$, a team $X$, and a formula $\phi$ whose free variables are included in $\text{dom}(X)$ is defined inductively as follows:
\begin{itemize}
    \item $\M\models_X p\iff \forall g\in X:\M\models_g p$ in standard Tarskian semantics\hfill if $p$ is an atom
    \item $\M\models_X\bot\iff X=\emptyset$
    \item $\M\models_X\phi\land\psi\iff\M\models_X\phi$ and $\M\models_X\psi$
    \item $\M\models_X\phi\lori\psi\iff\M\models_X\phi$ or $\M\models_X\psi$
    \item $\M\models_X\phi\to\psi\iff\forall Y\subseteq X:\M\models_Y\phi$ implies $\M\models_Y\psi$
    \item $\M\models_X\forall x\phi\iff\forall d\in D:\M\models_{X[x\mapsto d]}\phi$
    \item $\M\models_X\existsi x\phi\iff\exists d\in D:\M\models_{X[x\mapsto d]}\phi$
    \item $\M\models_X[x]\phi\iff\M\models_{X[x\mapsto D]}\phi$
\end{itemize}
\end{definition}
\noindent
For a detailed discussion of the clauses for atoms and connectives, we refer to \cite{Ciardelli:23book}. As for the quantifiers, $\forall x$ and $\existsi x$ test different ways of assigning to the variable $x$ a constant value, and check whether the embedded formula is supported in all/some of the resulting teams. By contrast, $[x]$ checks the support of its argument relative to a single team, obtained by assigning to $x$ all possible values. For
discussion of the philosophical difference between the kinds of generality achieved by $\forall x$ and $[x]$, see \cite{Ciardelli:26}.
The following general properties are familiar from the literature on inquisitive logic and dependence logic.

\begin{proposition}[Persistency] For any model $\M$, teams $X,Y$, and formulas $\phi$ of \inqBTe, if $\M\models_X\phi$ and $Y\subseteq X$ then $\M\models_Y\phi$.
\end{proposition}

\begin{proposition}[Empty team property] For any model $\M$, $\M\models_\emptyset\phi$ for all formulas $\phi$ of \inqBTe.
\end{proposition}

\begin{proposition}[Locality] For any model $\M$, any formula $\phi$ of $\inqBTe$, and all teams $X,Y$ whose domain includes the set $\text{FV}(\phi)$ of free variables in $\phi$:
$$X|_{\text{FV}(\phi)}=Y|_{\text{FV}(\phi)}\text{ implies }\M\models_X\phi\iff\M\models_Y\phi,$$
where $X|_{\text{FV}(\phi)}$ denotes the team obtained by restricting each assignment in $X$ to $\text{FV}(\phi)$, and likewise for $Y$.
\end{proposition}

\noindent
Note that for a sentence $\phi$, all nonempty teams have the same restriction to $\text{FV}(\phi)=\emptyset$, namely, the singleton of the empty assignment. So by Locality, all nonempty teams agree about whether or not $\phi$ is supported. We can thus say that a model $\M$ satisfies $\phi$ (written $\M\models\phi$) in case $\M\models_X\phi$ holds relative to an arbitrary nonempty team $X$; if $\M\not\models\phi$ we say that $\M$ falsifies $\phi$. Classical propositional connectives behave standardly with respect to the relation of satisfaction, for instance, $\M\models\neg\phi\iff\M\not\models\phi$.

The logical notions of entailment, validity, and equivalence are defined in the obvious way: a set of formulas $\Phi$ entails a formula $\psi$ (denoted $\Phi\models\psi$), if in every model, any team that is defined on all the free variables in $\Phi\cup\{\psi\}$ and supports all $\phi\in\Phi$ also supports $\psi$; a formula $\phi$ is valid ($\models\phi$) if in every model it is supported by any team whose domain includes its free variables; two formulas $\phi,\psi$ are equivalent ($\phi\equiv\psi$) if they entail each other. As for the notion of satisfiability, given the empty state property, it must be defined with respect to nonempty teams only: $\phi$ is satisfiable if $\M\models_X\phi$ for some model $\M$ and some nonempty team $X$; note that if $\phi$ is a sentence, satisfiability simply means that $\phi$ is satisfied in some model, as usual.

The connection with standard first-order logic is made through an important semantic property called \emph{flatness} (or \emph{truth-conditionality} in the inquisitive literature).
We say that a formula $\phi$ is \emph{flat} if support for $\phi$ relative to a team reduces to support at each singleton sub-team.

\begin{definition}
    A formula $\phi$ is flat if for all models $\M$ and teams $X$ with $\text{dom}(X)\supseteq\text{FV}(\phi)$ we have
    $$\M\models_X\phi\iff\forall g\in X:\M\models_{\{g\}}\phi.$$
\end{definition}

\noindent
The following proposition shows that with respect to the classical fragment, our team semantics is essentially equivalent to the standard Tarskian semantics.

\begin{proposition}\label{prop:classical}
For classical formulas, support relative to a singleton team $\{g\}$ coincides with truth relative to $g$ in Tarskian semantics. Moreover, all classical formulas are flat. Thus, for a classical formula, support at a team $X$ coincides with truth relative to each $g\in X$ in Tarskian semantics. In particular, for a classical sentence, our definition of satisfaction in a model coincides with the one given by Tarskian semantics.
\end{proposition}

\noindent
Using this fact, it is easy to show that our logics are, in a precise sense, conservative extensions of classical first-order logic: for classical formulas, the notions of entailment, validity, and equivalence determined by our semantics coincide with those of classical first-order logic.

The property of flatness extends beyond the classical fragment to formulas that include $[x]$ but not the inquisitive operators $\lori$ and $\existsi$: indeed, any $\{\lori,\existsi\}$-free formula is equivalent to a classical one, since when $\alpha$ is flat, $[x]\alpha\equiv\forall x\alpha$. On the other hand, formulas including the inquisitive operators $\lori$ and $\existsi$ are typically not flat: they express global properties of a team which do not reduce to local properties of the individual assignments in the team. 

For a simple example, suppose $\alpha$ is a classical formula and consider the formula $?\alpha$ ( $:=\alpha\lori\neg\alpha$). It is easy to check that we have:
$$\M\models_X{?\alpha}\iff\alpha\text{ has the same truth value relative to each }g\in X,$$
where the right-hand side refers to the standard truth-value in standard Tarskian semantics. Intuitively, the formula $?\alpha$ is regarded as capturing the question whether $\alpha$ is true or false; it is supported if all assignments in the team agree on the answer to this question.
Obviously, $?\alpha$ is supported relative to any singleton team, but (typically) not relative to an arbitrary team, thus violating truth-conditionality.

Another example of an inquisitive formula, which will play an important role for our purposes, is the following. For a variable term $t$, let us define
$$\mu t\;:=\;\forall x?(x=t),$$
where $x$ is an arbitrary variable not occurring in $t$. It is easy to check that $\mu t$ is supported by a team $X$ just in case the value of $t$ is constant throughout $X$:
$$\M\models_X\mu t\iff \forall g,g'\in X:[t]_M^g=[t]_M^{g'}.$$
Intuitively, $\mu t$ expresses the question ``what is $t$''. We refer to formulas of the form $\mu t$ as \emph{value questions}.\footnote{It is easy to check that the same result can be obtained by the formula $\lambda t\,:=\,\existsi x(x=t)$. This was indeed the strategy adopted in some previous work \cite{Ciardelli:23book}. Here, we prefer to make use of the formulas $\mu t$ (defined in terms of $\forall$ and $?$) rather than $ \lambda t$ (defined in terms of $\existsi$)  since the former option allows us to obtain formulas expressing dependency that fit within the $\existsi$-free fragment of \inqBT, which is known to be especially well-behaved: as shown in \cite{CiardelliGrilletti:22} and \cite{CiardelliConti:26}, this fragment is entailment-compact and recursively axiomatizable.}
Using a strategy that goes back to \cite{abramsky}, we can then express dependencies in terms of value questions and implication. We have:
\begin{eqnarray*}
    \M\models_X\mu x_1\land\dots\land\mu x_n\to \mu y&\iff &\forall g,g'\in X: \;  \text{if }g(x_i)=g'(x_i)\text{ for }1\le i\le n\\
    &&\phantom{\forall g,g'\in X: }\;\text{ then }g(y)=g'(y).
\end{eqnarray*}
This is exactly the semantics of a dependence atom $=(\vec x,y)$ in dependence logic (where $\vec x=x_1\dots x_n$). For convenience, we will adopt the notation $=(\vec x,y)$ as an abbreviation, defined as follows:
$$=(\vec x,y)\quad:=\quad \mu x_1\land\dots\land\mu x_n\to \mu y.$$

\section{Results on \inqBT\ and \inqBTe}\label{results}

In this section we present our results for the logics $\inqBTe$ and $\inqBT$. 

\begin{theorem}\label{thm1} There is a sentence $\psi$ of $\inqBTe$ in the empty vocabulary which defines finiteness, i.e., such that for every model $\M$: $\M\models\psi\iff \M$ is finite. 
 \end{theorem}
\begin{proof}We utilize the fact that a set $A$ is finite if and only if every injection $f\colon A \rightarrow A$ is also a surjection. Recall that we write $\dep{x}{y}$ to abbreviate the formula $\mu x\to\mu y$, that is, $\forall z?(z=x)\to\forall z?(z=y)$, and similarly for $\dep{y}{x}$.
Define: 
\begin{eqnarray}
\label{phi}\phi(x,y)&:=&\big ( \dep{x}{y}\, \wedge \dep{y}{x}\land \existsi z (z\neq y)\; \rightarrow\;  \existsi u  (u\neq x)  \big),\\
\label{psi}\psi&:=&[x][y]\phi(x,y).
\end{eqnarray}
A model $\M=(D,I)$ falsifies $\psi$ just in case the team $X$ that consists of all assignments of values to $x$ and $y$ fails to support the formula $\phi(x,y)$. 
By the semantics of implication, this holds iff there is a sub-team $Y\subseteq X$ which supports the antecedent but not the consequent. As discussed above, a team $Y$ with domain $\{x,y\}$ can be identified with a binary relation $R=Y[x,y]=\{(g(x),g(y))\mid g\in Y\}$. Spelling out the semantics, we find that:
\begin{multicols}{2}
\begin{itemize}
\item $\M\models_Y \;\dep{x}{y}\iff R$ is a function;
\item $\M\models_Y \;\dep{y}{x}\iff R$ is injective;
\item $\M\not\models_Y\existsi u  (u\neq x)\iff \text{dom}(R)=D$;
\item $\M\models_Y\existsi z(z\neq y)\iff \text{ran}(R)\neq D$.
\end{itemize}
\end{multicols}
\noindent
Putting things together, then, $\psi$ is falsified in a model $\M=(D,I)$ just in case there exists a function defined on $D$ which is injective but not surjective, that is, just in case $D$ is infinite. Equivalently, $\psi$ is satisfied in a model $\M$ just in case $\M$ is finite.
\end{proof}

\noindent 
Since negation behaves classically at the level of sentences, our theorem also implies that the negation $\neg\psi$ of the sentence defined in \ref{psi} defines infinity, i.e., it is satisfied exactly by the infinite models. 
Since finiteness and infinity of the model are famously not expressible in first-order logic, our theorem immediately yields a negative answer to Open Question 2 in the introduction.

\begin{corollary} Some sentences of \inqBTe\ are not equivalent to any sentence of standard first-order logic.
\end{corollary}

\noindent
Using the possibility of expressing finiteness in \inqBTe\ we can show that this logic, like second-order logic, violates compactness and has a non-recursively enumerable (and, indeed, non-arithmetical) set of validities. First, recall that in the setting of inquisitive logic and team-based logics, it makes sense to distinguish between two notions of compactness.

\begin{definition} 
A logic $L$ is said to be:
\begin{itemize}

\item\emph{satisfiability compact}, if for all sets of $L$-formulas $\Phi$ the following holds: if  every finite  $\Phi_0\subset \Phi$ is satisfiable, then $\Phi$ is satisfiable; 

\item \emph{entailment compact}, if for all sets of $L$-formulas $\Phi$ and any formula $\psi$, the following holds: if $\Phi\models_L \psi$, then 
$\Phi_0\models_L \psi$, for some finite $\Phi_0 \subseteq \Phi$.
\end{itemize}
\end{definition}

\noindent
For a team-based logic including $\bot$, entailment compactness implies satisfiability compactness (since the satisfiability of $\Phi$ reduces to the non-entailment $\Phi\not\models\bot$) but, in general, not the other way around; for instance, dependence logic is satisfiability compact, but not entailment compact  as there is a sentence of dependence logic satisfied exactly by the infinite models \cite{Vaananen:2007:dependence}. For the logic \inqBTe, by contrast, even satisfiability-compactness fails.

\begin{theorem} The logic $\inqBTe$ is not satisfiability-compact (and, thus, not  entailment-compact).
\end{theorem}
\begin{proof}
Define $\Gamma =\{\phi_n\, |\, n\in \mathbb{N}\}$, where $\phi_n$ is a classical sentence expressing the existence of at least $n$ distinct objects: 
$$ \exists x_1\ldots \exists x_n \bigwedge_{i\not= j} (x_i\not= x_j). $$
Using Theorem \ref{thm1} and Proposition \ref{prop:classical}, we can see that the set $\Gamma\cup\{\psi\}$, where $\psi$ is defined as in \ref{psi}, is not satisfiable, but any finite subset of this set is satisfiable.
\end{proof}

\begin{theorem} The set of valid formulas of  $\inqBTe$ has non-arithmetic complexity.
\end{theorem}
\begin{proof}We first show that the standard model of arithmetic ${\mathcal{N}}=(\mathbb{N},+,\times, \le,0,1)$  can be characterized up to isomorphisms by a single sentence of $\inqBTe$. In order to define such a sentence, we take the conjunction of the usual Peano axioms without the induction scheme $\Phi_{PA}$ and conjoin it with a sentence that says that every element of the model  has only finitely many predecessors. This can be expressed  by adapting the sentence  $\psi$ from \eqref{psi} in the following way: 
$$\psi'\;:=\; \forall z [x]  [y] \big ( x\le z \;\wedge\; y \le z \;\wedge \dep{x}{y} \,\wedge \dep{y}{x}\,\land\; \existsi u (u\le z\;\land\; u\neq y)\;\rightarrow\;   \existsi t  (t\le z\; \land\; t\neq x)  \big).$$
Notice  that for any classical sentence $\theta$ it now holds that
$$ \mathcal{N}\models \theta \iff (\Phi_{PA}\wedge \psi'\to \theta) \textrm{ is valid in \inqBTe}. $$
Thus, the set of arithmetic truths reduces to the set of \inqBTe-validities through the computable map $\theta\mapsto (\Phi_{PA}\land\psi'\to\theta)$. Since the former set has non-arithmetic complexity, so does the latter.
\end{proof}

\noindent
Let us now turn to the inquisitive team logic \inqBT. Without the operator $[x]$, it is no longer possible to express the finiteness/infiniteness of the model; indeed, as we mentioned in the beginning, with respect to sentences \inqBT\ is no more expressive than standard first-order logic \cite{yang2014}. However, if we start from the sentence $[x][y]\phi(x,y)$ that we used to express finiteness and remove the prefix $[x][y]$, we obtain an open formula $\phi(x,y)$ of \inqBT. We will now see that this formula allows us to give a negative answer to Open Question 1 in the introduction, concerning the expressive power of open formulas in~\inqBT. 

\begin{theorem}\label{thm2} There is an open formula  $\phi(x,y)$ of  $\inqBT$ in the empty vocabulary that cannot be expressed in first-order logic over models with vocabulary $\{R\}$, where $R$ is a binary relation.
 \end{theorem}
\begin{proof}
Define $\phi(x,y)$ as in \ref{phi} above. 
We know from the proof of Theorem \ref{thm1} that, for any model $\M=(D,I)$, if $X_D$ is the maximal team over $D$ defined on the variables $x$ and $y$, we have 
\begin{equation}\label{eqn2}
\M\models _{X_D} \phi(x,y)\iff D\text{ is finite}.
\end{equation}
Assume, for a contradiction, that there exists a first-order sentence $\phi^*(R)$, in the vocabulary comprising only the relation symbol $R$, satisfying 
\begin{equation}\label{eqn3}
 \M \models _X \phi(x,y) \iff (\M,X[x,y])\models \phi^* (R),  
 \end{equation}
for all models $\M$ and teams $X$. 
In particular, for any model $\M$ with domain $D$, if we instantiate $X$ to the maximal team $X_D$ for the variables $x,y$, then given that $X_D[x,y]=D^2$, from (\ref{eqn2}) and (\ref{eqn3}) we obtain:
\begin{equation}\label{eqn4}
(\M,D^2)\models \phi^* (R)\iff D\text{ is finite}.
 \end{equation}
Now let $k$ be the quantifier-rank of $\phi^*(R)$ and let $K=\{1,2,\ldots,k\}$. We can show by a simple Ehrenfeucht-Fra\"iss\'e game argument that for all first-order sentences $\theta$ of quantifier-rank up to $k$ we have:
\begin{equation}\label{equiv}
(K,K^2) \models \theta \iff (\mathbb{N},\mathbb{N}^2)\models \theta.
\end{equation}
Indeed, a winning strategy for the Duplicator in the $k$-round game over the models $(K,K^2)$ and $(\mathbb{N},\mathbb{N}^2)$ consists in maintaining the condition that 
$$ a_i=a_j \iff b_i=b_j, $$ 
for all elements $a_1,\ldots,a_h \in K$ and $b_1,\ldots,b_h \in \mathbb{N}$ chosen at any given stage in the game. This condition can be maintained for all $h\le k$---and, so, until the end of the game---since the cardinality of both models is at least $k$, and so at any stage in the game, Duplicator can always find a fresh element if needed.

Since the quantifier-rank of $\phi^*(R)$ is $k$, the equivalence  \eqref{equiv} holds for it.
But this is a contradiction, since (\ref{eqn4}) implies that $(K,K^2)\models\phi^*(R)$ but $(\mathbb{N},\mathbb{N}^2)\not\models\phi^*(R)$.
\end{proof}

To conclude our investigation of $\inqBT$, we examine its expressive power over finite structures. We show that $\inqBT$ can express properties of finite structures and teams that cannot be expressed in first-order logic. This is established by constructing an $\inqBT$ formula whose model checking problem is provably harder than that of any first-order formula.

Let $\threeSAT$ (respectively, $\threeUNSAT$) denote the set of satisfiable (respectively, unsatisfiable) Boolean formulas in conjunctive normal form with clauses of size three. It is known that $\threeUNSAT$ is $\mathrm{coNP}$-complete, and that the model checking problem of any first-order formula has a strictly lower computational complexity (see, e.g., \cite{NI99}).
\begin{theorem}
There exists a formula $\varphi$ of $\inqBT$ such that the problem of deciding whether a given finite structure and  a team satisfies $\varphi$ is $\mathrm{coNP}$-complete.
\end{theorem}

\begin{proof}
We give a reduction from $\threeUNSAT$. 
 Let 
\[
\Theta := \bigwedge_i (l_{i_1}\vee l_{i_2}\vee l_{i_3})
\]
be a Boolean formula, where each $l_{i_j}$ is either a propositional symbol ($p_k$) or its negation ($\neg p_k$). 
We use an encoding of $\Theta$ as a finite structure and a team that is similar to the one used for showing the $\mathrm{NP}$-completeness of the model checking of certain simple quantifier-free formulas of  dependence logic \cite{kontinenj13}.

Define a team $X$ over the domain $\{z,u,x,y\}$ as follows.
Each clause $c_i=(l_{i_1}\vee l_{i_2}\vee l_{i_3})$ of $\Theta$
is represented by three assignments:

\begin{center}
\begin{tabular}{|c|c|c|c|} \hline
$z$ = clause & $u$ = position & $x$ = variable & $y$ = parity \\ \hline\hline
$i$ & a & $p_{i_1}$ & $\epsilon_{i_1}$ \\ \hline
$i$ & b & $p_{i_2}$ & $\epsilon_{i_2}$ \\ \hline
$i$ & c & $p_{i_3}$ & $\epsilon_{i_3}$ \\ \hline
\end{tabular}
\end{center}

Here $\epsilon_{i_j}=+$ if the literal $l_{i_j}$ is positive
and $\epsilon_{i_j}=-$ if it is negated.
Let $X$ be the union of all these assignments. The domain $M$ of the associated structure $\M_{\theta}$ consists exactly of the elements occurring as values of variables in $X$. We assume that the possible values of different variables in $X$ partition the domain of $M$ and we use unary relations $V$ and $C$ to distinguish the elements encoding variables and clauses of $\Theta$ from the rest of the elements of $M$.  

We use the following idea:  any subteam $Y\subseteq X$ that satisfies $\dep{x}{y}$ and includes, for each clause $c_i$ of $\Theta$, at least one assignment $s$ such that $s(z)=i$ 
encodes a (set of)  Boolean assignments satisfying $\Theta$. One of them can be defined as follows:
\[
f(p_i)=1 \;\Longleftrightarrow\;
\exists s\in Y \text{ with } s(x)=p_i \text{ and } s(y)=+.
\]
The dependence atom $\dep{x}{y}$ guarantees that $f$ is well-defined. Furthermore, it is easy to check that any satisfying assignment $f$ for $\Theta$ gives rise to a subset $Y\subseteq X$ that satisfies $\dep{x}{y}$ and includes at least one assignment for each clause.
Therefore, in order to express that $\Theta$ is  unsatisfiable, it suffices to express that any subset $Y\subseteq X$ satisfying $\dep{x}{y}$ necessarily excludes, for some $i$, all $s$ such that $s(z)=i$. This can be expressed by the following formula:
\[
\varphi \; :=\; 
(\dep{x}{y} \to 
\existsi w \bigl(C(w)\wedge w\not=z )).
\]
It is now straightforward to verify that indeed $\Theta$ is  unsatisfiable, if and only if, $ \M_{\theta}\models_X \varphi$. Furthermore, since the model $\M_{\theta}$ and the team $X$ can be constructed in polynomial time in the size of $\Theta$, the model checking problem $\M\models_Y \varphi$ is $\mathrm{coNP}$-hard.
Membership in $\mathrm{coNP}$ follows from the fact that $\varphi$ can be easily rewritten in second-order logic using universal quantification over relations assuming the team $X$ is available as an auxiliary relation $X[z,u,x,y]$. 
\end{proof}

\section{Extension to \inqBQ}
\label{sec:inqbq}
In this section we turn our attention to the standard system of inquisitive first-order logic, \inqBQ. We will show that a simple adaptation of the proof we gave above for \inqBT\ allows us to settle Open Question~3 about the expressive power of \inqBQ\ in the negative.

\paragraph{Basic notions.} The language of \inqBQ\ is the same as for \inqBT. However, \inqBQ\ is interpreted over \emph{first-order information models}, which are triples $M=(W,D,I)$, where $W$ is a non-empty set whose elements are called \emph{worlds}, $D$ is a non-empty set whose elements are called \emph{individuals}, and $I$ is a world-relative interpretation function that gives for each $w\in W$ a map $I_w$ assigning to each $n$-ary relation symbol $R$ an $n$-ary relation $I_w(R)$ over $D$, and to each $n$-ary function symbol $f$ an $n$-ary function $I_w(f)$ over $D$ (instead of $I_w(R)$ and $I_w(f)$ we also write $R_w$ and $f_w$). Thus, with each world $w$ we can associate a standard relational structure $\M_w=(D,I_w)$, and the model $M$ as a whole can be regarded as a family $(\M_w)_{w\in W}$ of relational structures  over a common domain. 

An \emph{information state} (or simply \emph{state}) in $M$ is a subset $s\subseteq W$. An assignment is a function $g$ from the set of variables into $D$.
Terms $t$ from the language are assigned a denotation $[t]_w^g$ relative to a world $w$ and an assignment $g$ inductively by letting $[x]_w^g=g(x)$ and $[f(t_1,\dots,t_n)]_w^g=f_w([t_1]_w^g,\dots,[t_n]_w^g)$. The semantics is given in terms of a relation of \emph{support} $M,s\models_g\phi$ that interprets a formula $\phi$ relative to a model $M$, a state $s$, and an assignment $g$. This relation is defined by the following clauses:
\begin{itemize}
\item $M,s\models_g R(t_1,\dots,t_n)\iff\forall w\in s:([t_1]_w^g,\dots,[t_n]_w^g)\in R_w$
\item $M,s\models_g (t_1=t_2)\iff\forall w\in s: [t_1]_w^g=[t_2]_w^g$\footnote{The clause we consider here is a special case: in general, the identity symbol can be treated in \inqBQ\ as a world-dependent congruence relation \cite{Ciardelli:23book}. Models where `=' is interpreted as meta-language identity are called \textsf{id}-models. For our purposes, we can restrict to this special class of models: if a formula $\phi$ does not have a first-order counterpart relative to a more restricted class of models, \emph{a fortiori} it does not have such a counterpart relative to a more general class of models.}
\item $M,s\models_g\bot\iff s=\emptyset$
\item $M,s\models_g\phi\land\psi\iff M,s\models_g\phi$ and $M,s\models_g\psi$
\item $M,s\models_g\phi\lori\psi\iff M,s\models_g\phi$ or $M,s\models_g\psi$
\item $M,s\models_g\phi\to\psi\iff\forall t\subseteq s: M,t\models_g\phi$ implies $M,t\models_g\psi$
\item $M,s\models_g\forall x\phi\iff\forall d\in D: M,s\models_{g[x\mapsto d]}\phi$
\item $M,s\models_g\existsi x\phi\iff\exists d\in D: M,s\models_{g[x\mapsto d]}\phi$
\end{itemize}
We write simply $M\models_g\phi$ in case $M,W\models_g\phi$. As usual, for sentences the assignment does not matter and we can drop reference to it.

Clearly, the semantics of \inqBQ\ is strictly parallel to that of \inqBT, where states now play the role of teams. As a consequence, these logics share many properties. In particular, support in \inqBQ\ is persistent (i.e., preserved under substates), and the empty state trivially supports every formula. Classical formulas are flat, i.e., they are supported at a state $s$ just in case they are supported at each singleton $\{w\}\subseteq s$. Moreover, for a classical formula support at a singleton $\{w\}$ coincides with truth at the relational structure $\M_w=(D,I_w)$ under Tarskian semantics. 

\paragraph{Relational encoding.} A first-order information model $M=(W,D,I)$ for a signature $\Sigma$ can be faithfully encoded as a relational model $M^*=(W,D,I^*)$ with two sorts, \texttt{w} for worlds and \texttt{e} for individuals, for a modified signature $\Sigma^*$ defined as follows:
\begin{itemize}
\item for every $n$-ary predicate symbol $P\in\Sigma$, $\Sigma^*$ contains a predicate symbol $P^*$ of arity $n+1$ where the first argument is of sort \texttt{w} and the remaining $n$ arguments are of sort \texttt{e};
\item for every $n$-ary function symbol $f\in\Sigma$, $\Sigma^*$ contains a function symbol $f^*$ of arity $n+1$ where the first argument is of sort \texttt{w} and the remaining $n$ arguments as well as the output are of sort \texttt{e}.
\end{itemize}
Given a first-order information model $M=(W,D,I)$ for the signature $\Sigma$, its encoding $M^*=(W,D,I^*)$ is the two-sorted model over the signature $\Sigma^*$ defined in the following way:
$$(w,\overline d)\in I^*(R^*)\iff  \overline d\in I_w(R),\qquad\qquad\qquad I^*(f^*)(w,\overline d)=I_w(f)(\overline d).$$
We are now in a position to formulate Open Question 3 from the introduction in a fully precise way.

\begin{itemize}
\item \textbf{Open Question 3.} Given a sentence $\phi$ of \inqBQ\ in a signature $\Sigma$, is there always a corresponding sentence $\phi^*$ of two-sorted predicate logic in the signature $\Sigma^*$ such that for any first-order information model $M$ we have $M\models\phi\iff M^*\models\phi^*$?
\end{itemize}
To get the intuitive idea, the table below includes some \inqBQ\ sentences and corresponding first-order sentences which are equivalent to them in the relevant sense. We consider a signature $\Sigma$ with a unary predicate $P$ and a constant symbol $a$, and use $w$ as a first-order variable of sort \texttt{w}.

\begin{center}
\begin{tabular}{ll}
\inqBQ\ sentence & first-order sentence\\
\hline
$P(a)$ & $\forall wP^*(w,a)$\\
$?P(a)$ & $\forall wP^*(w,a)\lor\forall w\neg P^*(w,a)$\\
$\existsi xP(x)$ & $\exists x\forall wP^*(w,x)$
\end{tabular}
\end{center}

\paragraph{Novel result.} 
Consider a signature $\Sigma$ containing two constant symbols, $a$ and $b$. A first-order information model for $\Sigma$ is a triple $M=(W,D,I)$ where relative to a world $w\in W$, the function $I$ assigns an extension for each constant, $a_w\in D$ and $b_w\in D$. Now to each information state $s\subseteq W$ we can associate a corresponding binary relation $R_s\subseteq D^2$ defined by $R_s=\{(a_w,b_w)\mid w\in s\}$. We say that a model $M=(W,D,I)$ is \emph{full} in case $R_W=D^2$. Now consider the following formula, obtained by taking the formula $\phi(x,y)$ defined in \ref{phi} and replacing the variables $x,y$ with the constants $a,b$:
\begin{equation}
\label{phi(a,b)}\phi(a,b)\;:=\;\big ( \dep{a}{b}\, \wedge \dep{b}{a}\land \existsi z (z\neq b)\; \rightarrow\;  \existsi u  (u\neq a)  \big).
\end{equation}
We prove that, in restriction to the class of full models, $\phi(a,b)$ defines finiteness of the domain.
\begin{proposition} If $M=(W,D,I)$ is a full model then $M\models\phi(a,b)\iff D$ is finite.
\end{proposition}

\begin{proof} We prove the equivalent claim that if $M$ is full, $M\not\models\phi(a,b)\iff D$ is infinite. We have that $M\not\models\phi(a,b)$ if there is an $s\subseteq W$ that supports the antecedent but not the consequent. As we discussed above, this state is associated with a relation $R_s=\{(a_w,b_w)\mid w\in s\}$. Spelling out the semantics, we find:
\begin{multicols}{2}
\begin{itemize}
\item $M,s\models{\dep{a}{b}}\iff R_s$ is a function
\item $M,s\models{\dep{b}{a}}\iff R_s$ is injective
\item $M,s\not\models\existsi u(u\neq a)\iff \text{dom}(R_s)=D$
\item $M,s\models\existsi z(z\neq b)\iff \text{ran}(R_s)\neq D$
\end{itemize}
\end{multicols}
\noindent
So, a state $s$ falsifies the conditional $\phi(a,b)$ just in case the associated relation $R_s$ is an injective function on $D$ which is not surjective. In a full model, every binary relation on $D$ is represented as $R_s$ for some state $s\subseteq W$. So, in a full model, a state falsifying the conditional exists iff there exists an injective function defined on $D$ which is not surjective, i.e., iff $D$ is infinite.
\end{proof}

\noindent
We can use this result to give a negative answer to Open Question 3.

\begin{theorem} For some sentence $\phi$ of \inqBQ\ in a signature $\Sigma$, there is no sentence $\phi^*$ of two-sorted first-order logic in the signature $\Sigma^*$ such that for every model $M$ of \inqBQ\ we have $M\models\phi\iff M^*\models\phi^*$.
\end{theorem}

\begin{proof} We can take $\phi$ to be the formula $\phi(a,b)$ defined in \ref{phi(a,b)}, in the signature $\Sigma$ including only the constants $a,b$. Towards a contradiction, suppose there is a formula $\phi^*$ of two-sorted first-order logic in the signature $\Sigma^*$ that is equivalent to $\phi$ in the relevant sense. 

Let $q$ be the quantifier rank of $\phi^*$, let $k=q^2$, and let $K=\{1,\dots,k\}$. Now consider the two models $M_K=(K^2,K,I_K)$ and $M_\N=(\N^2,\N,I_\N)$, where both $I_K$ and $I_\N$ assign to a world $(n,m)$ the extensions given by $a_{(n,m)}=n$ and $b_{(n,m)}=m$. Clearly, these models are full, so by the previous proposition we have $M_K\models\phi$ but $M_\N\not\models\phi$. By our assumption on $\phi^*$, it follows that $M_K^*\models\phi^*$ and $M_\N^*\not\models\phi^*$. If we make explicit the translated structures $M_K^*$ and $M_\N^*$, we find that they are:
$$M_K^*=(K^2,K,\pi_1^K,\pi_2^K)\qquad\qquad M_\N^*=(\N^2,\N,\pi_1^{\N},\pi_2^{\N}),$$
where $\pi_1^K$ and $\pi_2^K$ are the projection functions from $K^2$ to $K$, and similarly for $\pi_1^{\N},\pi_2^{\N}$. A standard argument shows that these structures are equivalent in a $q$-move Ehrenfeucht-Fra\"iss\'e game: the basic idea is that within $q$ moves, Spoiler can only pick up to $q$ distinct elements of $\N^2\cup \N$, corresponding to at most $q^2=k$ distinct natural numbers; Duplicator can then identify these numbers with corresponding numbers from $\{1,\dots,k\}$ and make the counterpart of Spoiler's moves under this identification. So, $M_K^*$ and $M_\N^*$ satisfy the same formulas of quantifier rank up to $q$. As a consequence, they must agree on $\phi^*$, contrary to our assumption.
\end{proof}

\noindent
We have thus established that some sentences of \inqBQ\ express genuine second-order properties of models: in particular, the formula $\phi(a,b)$ defined above is an example. 

It is interesting to comment on the syntactic form of this example. Recall from the introduction that, in contrast to our result, two broad fragments of \inqBQ\ allow for a translation to two-sorted first-order logic: the classical antecedent (clant) fragment \cite{Grilletti:21,Ciardelli:23book}, where antecedents are restricted to classical formulas, and the restricted existential (rex) fragment \cite{CiardelliGrilletti:22}, where the inquisitive existential quantifier is restricted to occur only in conditional antecedents. This implies that any example of an \inqBQ\ sentence expressing a non first-order property of models must have two features: it must involve (i) an inquisitive antecedent and (ii) an inquisitive existential which is not in a conditional antecedent. Our example sentence $\phi(a,b)$ respects both conditions in a relatively minimal way: it is a conditional with an inquisitive antecedent and an inquisitive existential in the consequent.

\section{Conclusion}\label{conc}
We have shown that there is a sentence of  $\inqBTe$ whose models are exactly the finite structures and that, therefore,  $\inqBTe$ fails  compactness and cannot be axiomatized. We also showed that 
 there are open formulas of $\inqBT$ which express non first-order properties of teams, and sentences of \inqBQ\ that express non first-order properties of models, viewed as two-sorted relational structures. Some important problems in this area remain open, including the problem of characterizing the exact expressive power of the logics $\inqBT$ and $\inqBQ$ in relation to fragments of second-order logic, and the problem of whether these logics satisfy entailment compactness and admit a recursive axiomatization.\footnote{In fact, after we completed this paper, we discovered that the ideas presented here can be used to settle these long-standing problems in the negative. A paper on this is in preparation; a first draft is available at  \url{https://arxiv.org/abs/2603.20845}.}
\subsection*{Acknowledgements}
The second author's research has been supported by the European Research Council under the Horizon Europe research and innovation programme (research project InqML, grant agreement no.\ 101116774).
\bibliographystyle{eptcs}
\bibliography{bibb}

\end{document}